\documentclass[11pt,reqno]{amsart}
\usepackage{fullpage}

\usepackage[english]{babel}
\usepackage[utf8x]{inputenc}
\usepackage[T1]{fontenc}
\usepackage{amsmath}
\usepackage{graphicx}
\usepackage{tikz}   
\usepackage{dsfont} 

\usepackage{amssymb}
\usepackage{amsthm}

\newtheorem{teo}{Theorem}[section]

\newtheorem{prop}{Proposition}[section]

\newcommand{\R}{\mathbb{R}}
\newcommand{\N}{\mathbb{N}}
\newcommand{\Rd}{{\mathbb{R}^{d}}}

\usepackage{tikz}
\usetikzlibrary{arrows,shapes}
\usetikzlibrary {positioning}
\tikzstyle{block}=[draw opacity=0.7,line width=1.4cm]
\usetikzlibrary{calc}
\usepackage{relsize}
\tikzset{fontscale/.style = {font=\relsize{#1}}
}

\graphicspath{{immagini/}}

\title{Multiple large-time behavior of nonlocal interaction equations with quadratic diffusion}
\author{M. Di Francesco, Y. Jaafra}
\address{M. Di Francesco, y. Jaafra - DISIM - Department of Information Engineering, Computer Science and Mathematics, University of L'Aquila, Via Vetoio 1 (Coppito)
	67100 L'Aquila (AQ) - Italy}
\email{marco.difrancesco@univaq.it}
\email{yahyaya.jaafra@graduate.univaq.it}
\date{}

\begin{document}
	
	\begin{abstract}
		In this paper we consider a one-dimensional nonlocal interaction equation with quadratic porous-medium type diffusion in which the interaction kernels are attractive, nonnegative, and integrable on the real line. Earlier results in the literature have shown existence of nontrivial steady states if the $L^1$ norm of the kernel $G$ is larger than the diffusion constant $\varepsilon$. In this paper we aim at showing that this equation exhibits a "multiple" behavior, in that solutions can either converge to the nontrivial steady states or decay to zero for large times. We prove the former situation holds in case the initial conditions are concentrated enough and "close" to the steady state in the $\infty$-Wasserstein distance. Moreover, we prove that solutions decay to zero for large times in the diffusion-dominated regime $\varepsilon\geq \|G\|_{L^1}$. Finally, we show two partial results suggesting that the large-time decay also holds in the complementary regime $\varepsilon< \|G\|_{L^1}$ for initial data with large enough second moment. We use numerical simulations both to validate our local asymptotic stability result and to support our conjecture on the large time decay.
	\end{abstract}
	
	\maketitle
	
	\section{Introduction}
	
	Several phenomena in biology are governed by the combination of long-range attractive effects with short-range repulsive ones. Examples range from chemotaxis of cells (see e.g. \cite{blanchet}) to swarming phenomena in animal biology \cite{Morale2005,capasso,mogilner,Topaz2006}. These situations can be modeled at either particle (microscopic) or continuum (macroscopic) level. In the former case one assumes that in a given finite set of particles (or individuals) each pair of particles interact through a "drift" velocity field depending on their distance, consisting of \emph{attractive} and \emph{repulsive} forces. What characterizes our approach is that the attractive force maintains a long range scale while passing to the continuum regime, whereas the range of interaction in the repulsive terms degenerates as the number of individuals becomes large. Such approach leads to our prototype model
	\begin{equation}\label{eq:main_intro}
	\partial_{t}\rho=\mathrm{div}(\rho\nabla(\varepsilon\rho
	-G*\rho)) \hspace{1.5cm} (x,t)\in \Rd\times [0,+\infty),
	\end{equation}
	in which the unknown $\rho=\rho(x,t)$ is sought in the set of time dependent curves from $t\in [0,+\infty)$ onto the space of nonnegative $L^1$ densities with fixed mass. 
	
	In \eqref{eq:main_intro} the quadratic diffusion term models (local) repulsive effects via a diffusion constant $\varepsilon>0$, and the nonlocal attractive term is governed by an interaction potential $G:\Rd\rightarrow[0,+\infty)$ satisfying $G(x)=g(|x|)$ for some $C^1$ function $g:[0,+\infty)\rightarrow [0,+\infty)$ with $g'<0$. The latter assumption in particular implies that the nonlocal drift term in \eqref{eq:main_intro} is \emph{attractive}, in that it yields a decrease in time of all moments $\int |x|^p \rho(x,t)\, dx$ of any order $p\in [1,+\infty)$. As a consequence of that, $\rho$ is expected to concentrate to a Dirac's delta centered at the initial center of mass as $t\rightarrow +\infty$ in case $\varepsilon = 0$. On the other hand, the diffusion part exerts an opposed effect on $\rho$, in that it implies a growth of all moments, and it would lead the solution to zero as $t\rightarrow +\infty$ in case $G\equiv 0$. Models of the form \eqref{eq:main_intro} can be recovered not only in biology as mentioned above, but also in material sciences (e.g. granular media models \cite{toscani_granular} and crystal defects modeling \cite{peletier}) and in social sciences, see for example pedestrian movements \cite{helbing} and opinion formation modeling \cite{sznajd}. The latter in particular justifies the derivation of PDE from a discrete set of equation, which can be stochastic of deterministic depending on the context. 
	
	The mathematical theory for \eqref{eq:main_intro} heavily relies on the associated \emph{energy functional}
	\begin{equation}\label{energy_intro}
	E[\rho]=\frac{\varepsilon}{2}\int_{\R^d}\rho^{2}dx-\frac{1}{2}\int_{\R^d}\rho G*\rho dx
	\end{equation}
	defined on the space of probability measures $\mathcal{P}(\R^d)$ (the total mass is preserved in time in \eqref{eq:main_intro}) with finite $L^2$-norm, with the obvious extension to $+\infty$ when $\rho\in \mathcal{P}(\R^d)\setminus L^2(\R^d)$. The energy functional $E$ can be used to interpret \eqref{eq:main_intro} as a Wasserstein gradient flow in the sense of \cite{AGS,JKO}. Assuming that $G$ is $\lambda$-convex as a function on $\R^d$ for some $\lambda\in \R$ (i.e. $G(x)+\frac{\lambda}{2}|x|^2$ convex), then the functional $E$ is $\mu$-displacement convex in the sense of \cite{mccann,AGS} with $\mu=\min\{2\lambda,0\}$. The existence (and uniqueness) provided by the so-called JKO scheme \cite{JKO} is global-in-time in $L^2$ provided $G$ is smooth enough. Unlike the $2d$ Keller-Segel model for chemotaxis (see e.g. \cite{jager,blanchet}), no concentrations to Dirac's deltas occur in finite time if $G$ is has (for example) two continuous and bounded derivatives. The case $\varepsilon=0$ has been studied by many authors \cite{li2004long,bodnar,burger2008large,fellner,bertozzi1,bertozzi2,bertozzi3}. 
	
	As solutions cannot blow up in finite time when $G$ is smooth, the dichotomy between the repulsive diffusive effect and the attractive effect induced by the nonlocal term is resolved in terms of the large time behavior: a diffusive behavior would result in the large time decay of solutions as $t\rightarrow +\infty$, for example w.r.t the $L^2$ norm, whereas the aggregative behavior would yield the formation of a pattern, i.e. a (nontrivial) steady state. Such a problem has been studied mainly in two papers in the past years. In \cite{Bedrossian20111927} the problem has been addressed at the level of the existence of nontrivial minimizers for the energy functional $E[\rho]$. Almost at the same time, \cite{BurDiFFra12} addressed the existence and uniqueness of nontrivial steady states for \eqref{eq:main_intro}. The combination of the two results provide the following general picture in the special case in which $G$ is \emph{nonnegative} and \emph{integrable}:
	\begin{itemize}
		\item If $\varepsilon \geq \|G\|_{L^1(\R^d)}$, then the only stationary state for \eqref{eq:main_intro} is $\rho\equiv 0$. No global minimizers exist under the fixed mass constraint $\int \rho \, dx = M>0$.\footnote{It follows as a special case of the results in \cite{Bedrossian20111927}. The critical case $\varepsilon=\|G\|_{L^1}$ is treated in \cite{BurDiFFra12}.}
		\item If $0\leq \varepsilon < \|G\|_{L^1(\R^d)}$, then the functional $E$ has a nontrivial global minimizer in $L^1(\R^d)$ under the constraint $\int \rho \, dx = M>0$. Moreover, in case $G$ is supported on $\R$, the equation \eqref{eq:main_intro} has a unique nontrivial stationary solution up to multiplications by a constant and translations, which coincides with the global minimizer for $E$. \footnote{The existence of minimizers follows as a byproduct of the results in \cite{Bedrossian20111927}. The uniqueness result in one space dimension was first proved in  \cite{BurDiFFra12}. The multidimensional uniqueness result is due to \cite{kaib}.}
	\end{itemize}
	Partial extensions of the previous results to the case of a porous medium term $\mathrm{div} \rho\nabla \rho^{m-1}$ were proved in \cite{burger_fetecau_wang,choksi,kaib}. Let us mention at this stage that an extensive (and very deep) literature has been produced for the existence of global minimizers of the energy $E$ in which the interaction potential is not integrable and featuring confining property at infinity. For the case of Newtonian potentials we mention \cite{lions84,lieb_yau,blanchet_carrillo_laurencot,kim_yao,volzone}. For more general kernels we mention \cite{calvez_carrillo_hoffmann}.
	
	Once the existence or non-existence of nontrivial steady states is clear, a natural question arises on whether or not those steady states are attractors for the semi-group \eqref{eq:main_intro}. Even in the case of potentials with confining properties, the situation is not completely clear except in the case of "smooth" power laws $G(x)=|x|^\gamma$, with $\gamma\geq 1$, see \cite{carrillo_mccann_villani}. We refer to the introduction of \cite{calvez_carrillo_hoffmann} for a very clear and detailed explanation. To our knowledge, the only result which deals with this issue is \cite{bedrossian2014}, in which the large time decay of solutions is proved for initial data with large second moment and large mass, with potentials essentially behaving as the Bessel potential in $\R^d$ and in high enough dimension. The problem of detecting the large time behavior seems more difficult in the case of \eqref{eq:main_intro}, as the quadratic homogeneity in the energy \eqref{energy_intro} does not allow to play with the initial mass in order to penalize one of the two terms in the right-hand side of \eqref{eq:main_intro}. Moreover, having to do with an integrable kernel under pretty general assumptions does not allow to use the homogeneity of the kernel and rescale the equation, as often done in the case of confining interaction potentials.
	
	In our paper we address the problem of the large time behavior of \eqref{eq:main_intro}. Our results are restricted to the one-dimensional case and to the case of smooth potentials. More precisely, we assume
	\begin{itemize}
		\item[(1)] $G \geq 0,$ and $\mathrm{supp}(G)=\R,$
		\item[(2)] $G \in L^{1}(\R) \cap L^{\infty}(\R) \cap C^{2}(\R),$
		\item[(3)] $G(x)=G(-x)$ for all $x\in \R,$
		\item[(4)] $G^{\prime\prime}(x)< -c <0$ on $[-\lambda,\lambda]$ for some $\lambda, c > 0$, 
		\item[(5)] $G^{\prime}(x)<0$ for all $x>0$.
	\end{itemize}
For simplicity and without restriction, we shall assume throughout the paper that $\|G\|_{L^1}=1$. Next we summarize the structure of the paper as well as our main results.
	\begin{itemize}
		\item We first prove the unique steady state provided by \cite{BurDiFFra12} in the case $\varepsilon\in (0,1)$ is locally asymptotically stable in the $2$-Wasserstein distance. The result is stated in Theorem \ref{mainthm}.
		\item We prove in Theorem \ref{T2} that all solutions with finite energy $E$ decay to zero locally in $L^2$ and almost everywhere in $x\in \R$ as $t\rightarrow +\infty$ in case $\varepsilon\geq 1$.
		\item In section \ref{sec:arguments} we provide some (incomplete) arguments suggesting that the large time decay may occur also in the case $\varepsilon\in (0,1)$ for suitable initial conditions. 
		\item In section \ref{sec:numerics} we produce some numerical simulation to support our conjecture that \eqref{eq:main_intro} with $\varepsilon\in(0,1)$ features a multiple behavior for large times, i.e. that there are more than one attractors for \eqref{eq:main_intro} as $t\rightarrow +\infty$.
	\end{itemize}

	\section{Local Stability of Steady States for Smooth Attractive Potentials}\label{sec:stationary}
	
	In this section we study the long-time behavior of the solution to the one-dimensional evolution equation
	\begin{equation}
	\label{assymptotic1}
	\partial_{t}\rho=\partial_{x}(\rho\partial_{x}(\varepsilon\rho
	-G*\rho)) \hspace{1.5cm} \R\times \R^{+}.
	\end{equation}
	We do the whole analysis on a new formulation of the evolution equation \eqref{assymptotic1} obtained by using the Wasserstein metric in one dimensional space.
	
	We consider equation \eqref{assymptotic1} where the unknown $\rho(.,t)$ is a time-dependent probability density on $\R$, $\varepsilon$ is a fixed constant in $(0,\|G\|_{L^1}),$ and $G$ is the aggregation kernel with the following assumptions
	
	\begin{itemize}
		\item[(1)] $G \geq 0,$ and supp$(G)=\R,$
		\item[(2)] $G \in W^{1,1}(\R) \cap L^{\infty}(\R) \cap C^{2}(\R),$
		\item[(3)] $G(x)=g(|x|)$ for all $x\in \R,$
		\item[(4)] $G^{\prime\prime}(x)< -c <0$ on $[-\lambda,\lambda]$ where $\lambda, c > 0$, 
		\item[(5)] $g^{\prime}(r)<0$ for all $r>0,$
		\item[(6)] $\lim_{r\rightarrow +\infty}g(r)=0.$
	\end{itemize}
	
	In the following we review briefly the Wasserstein metric and we shall see how one can reformulate \eqref{assymptotic1} using a simplified expression of the $p-$Wasserstien distance obtained in one space dimension written in terms of pseudo inverses of the cumulative distributions of some probability measures.

	
	Let $\mathcal{P}(\Rd)$ be the space of the probability measures on $\Rd$. We denote by $\mathcal{P}_{p}(\Rd)$ the space of probability  measures $\mu\in \mathcal{P}(\Rd)$ having a finite $p-$moment $\int_{\Rd}|x|^{p}d\mu(x) < +\infty$. Then for two probability measures $\mu_{1}$ and $\mu_{2}$ in $\mathcal{P}_{p}(\Rd),$ the $p-$Wasserstein distance between them is defined by 
	
	\begin{equation}
	\label{assymptotic4}
	W_{p}(\mu_{1},\mu_{2})^{p}=\inf \bigg\{ \int\int_{\Rd\times\Rd} |x-y|^{p}d\pi(x,y), \pi\in\Pi(\mu_{1},\mu_{2}) \bigg\},
	\end{equation}
	where $\Pi(\mu_{1},\mu_{2})$ is the space of all measures $\pi$ on the product space $\Rd\times\Rd$ having $\mu_{1}$ and $\mu_{2}$ as marginals, i.e.
	
	\begin{equation*}
	\int\int_{\Rd\times\Rd} f(x_{i})d\pi(x_{1},x_{2})= \int_{\Rd}f(x_{i})d\mu_{i}(x_{i}),
	\end{equation*}
	for any $\mu_{i}-$integrable Borel function $f$, $i=1,2.$


	In the one dimensional space the $p-$Wasserstein distance can be rewritten by a different expression than \eqref{assymptotic4} which simplifies the analysis. In particular, the $p-$Wasserstein distance between $\mu_{1}$ and $\mu_{2}$ can be written in one space dimension as the $L^{p}-$ difference between the pseudo inverses of the cumulative distributions of $\mu_{1}$ and $\mu_{2}$. More precisely, let $F_{i} : \R \rightarrow [0,1]$, $i=1,2$, be the distribution function defined by
	\begin{equation}
	\label{assymptotic6}
	F_{i}(x) = \mu_{i}((\-\infty,x]), \hspace{2ex} i=1,2.
	\end{equation}
	Then the pseudo-inverse function $u_{i}:[0,1]\rightarrow \R$ of $F_{i}$ is defined by
	\begin{equation}
	\label{assymptotic7}
	u_{i}(z):= F_{i}^{-1}(z)= \inf\{x\in \R \hspace{1ex}| \hspace{1ex} F_{i}(x)>z\}, \hspace{3ex} i=1,2.
	\end{equation}
	Using these notations, the $p-$Wasserstein distance between $\mu_{1}$ and $\mu_{2}$ has the following new expression (see \cite{villani})
	\begin{equation}
	\label{assymptotic8}
	W_{p}(\mu_{1},\mu_{2})= \| u_{1}-u_{2}\|_{L^{p}([0,1])}.
	\end{equation}

	
	Following the procedure in the seminal papers \cite{carrillo2005wasserstein,russo1990deterministic}, we now rewrite the evolution equation \eqref{assymptotic1} in terms of the pseudo inverse function. Assume for simplicity the solution $\rho(t)$ to \eqref{assymptotic1} is smooth, positive, and has a connected compact support. Let $F(t)$ be its cumulative distribution function. Then one can easily show that its inverse $u(t):[0,1] \rightarrow \R$ satisfies the following equation
	\begin{equation}
	\label{assymptotic9}
	\partial_{t} u = -\frac{\varepsilon}{2} \partial_{z}\big( (\partial_{z}u)^{-2} \big) + \int_{0}^{1} G^{\prime}\big( u(z,t)-u(\xi,t) \big)d\xi, \hspace{2.5ex} (z,t)\in [0,1]\times[0,+\infty).
	\end{equation}
	Then, thanks to the expression \eqref{assymptotic8}, we can study the $p-$Wasserstein distance between the generic solution and the steady state of \eqref{assymptotic1} by direct computations on the time evolution of the $L^{p}$ distance of the difference between the solution and the steady state to the pseudo inverse equation \eqref{assymptotic9}. In principle, such a computation would  require enough regularity on the pseudo-inverse function and a compact support for all times. However, a standard approximation procedure can be implemented to bypass this problem, see \cite{gualdani} \cite{li2004long} - in which this technique was used for the first time - and \cite{burger2008large}, in which the initial condition gets approximated by a strictly positive, smooth, bounded density and a zero-flux condition is imposed on a large interval including the support of the solution at all times. The fact that the support remains compact is a consequence of \cite[Theorem 2.12]{burger2008large}. We omit the details, for which we refer to the aforementioned references.

	
	We consider the time evolution of the Wasserstein distance \eqref{assymptotic8} between a generic solution and the stationary one of equation \eqref{assymptotic1}. In particular, we will prove that the $2$-Wasserstein distance between the generic solution and the steady state to equation \eqref{assymptotic1} goes exponentially to zero with respect to time under some assumptions on the initial datum and the steady state. More precisely, we take the initial datum to be close to the steady state, in particular we consider "confined" initial data.
	
	Let $\rho_{\infty}(x)$ denote the unique stationary solution to equation \eqref{assymptotic1} with unit mass and zero center of mass, and let $u_{\infty}(z)$ denote the pseudo inverse of its cumulative distribution. Consider equation \eqref{assymptotic1} with an initial datum $\rho(x,0)=\rho_{0}(x)$ with unit mass and zero center of mass, and let $u_{0}(z)$ be the corresponding pseudo inverse. Then let $u$ be the corresponding solution to \eqref{assymptotic9}. The time evolution of the Wasserstein distance $\| u(z,t)-u_{\infty}(z)\|_{L^{2k}([0,1])}^{2k}$ where $k\in\N,$ is easily found to satisfy
	
	\begin{equation}
	\label{evolution}
	\begin{aligned}
	\frac{d}{dt} \| u(z,t)-u_{\infty}(z)\|_{L^{2k}([0,1])}^{2k} &= 2k \int_{0}^{1} [u(z,t)-u_{\infty}(z)]^{2k-1} \partial_{t}u(z,t)\, dz
	\\
	& = -k\varepsilon \int_{0}^{1} [u(z,t)-u_{\infty}(z)]^{2k-1} \partial_{z}((\partial_{z}u)^{-2})(z,t)\,dz 
	\\
	& \qquad\qquad\qquad + 2k\int_{0}^{1}\int_{0}^{1}[u(z,t)-u_{\infty}(z)]^{2k-1}G^{\prime}(H(z,\xi))d\xi dz
	\end{aligned}
	\end{equation} 
	where, to simplify notation, we set
	
	\begin{equation}
	\label{H}
	H(z,\xi;t):= u(z,t)-u(\xi,t).
	\end{equation}
	Then since $u_{\infty}$ is a stationary solution, it satisfies 
	
	\begin{equation}
	\label{Stationary}
	0 = -\frac{\varepsilon}{2} \partial_{z}((\partial_{z}u_{\infty})^{-2}) + \int_{0}^{1} G^{\prime}(u_{\infty}(z)-u_{\infty}(\xi))d\xi.
	\end{equation}
	Using \eqref{Stationary}, \eqref{evolution} becomes 
	
	\begin{equation}\label{I1I2}
	\begin{aligned}
	&  \frac{d}{dt} \| u(z,t)-u_{\infty}(z)\|_{L^{2k}([0,1])}^{2k} =
	\\
	& = -k\varepsilon \int_{0}^{1} [u-u_{\infty}]^{2k-1} \Big[ \partial_{z}((\partial_{z}u)^{-2}) - \partial_{z}((\partial_{z}u_{\infty})^{-2}) \Big] dz 
	\\
	&\qquad + 2k\int_{0}^{1}\int_{0}^{1} [u-u_{\infty}]^{2k-1} \Big[ G^{\prime}(H(z,\xi)) - G^{\prime}(K(z,\xi)) \Big] d\xi dz
	\\
	& := I_{1} + I_{2}
	\end{aligned}
	\end{equation} 
	with $H(z,\xi)$ defined by \eqref{H} and 
	
	\begin{equation}
	\label{K}
	K(z,\xi):= u_{\infty}(z)-u_{\infty}(\xi).
	\end{equation}
	Now, we analyze the two integrals $I_{1}$ and $I_{2}$ separately
	
	\begin{equation}
	\label{I1}
	\begin{aligned}
	I_{1} & =  -k\varepsilon \int_{0}^{1} [u-u_{\infty}]^{2k-1}\partial_{z} \big[ (\partial_{z}u)^{-2} - (\partial_{z}u_{\infty})^{-2} \big] dz
	\\
	& = \varepsilon k(2k-1)\int_{0}^{1} [u-u_{\infty}]^{2k-2} [\partial_{z}u - \partial_{z}u_{\infty}] \big[ (\partial_{z}u)^{-2} - (\partial_{z}u_{\infty})^{-2} \big] dz
	\\
	& = \varepsilon k(2k-1)\int_{0}^{1} [u-u_{\infty}]^{2k-2} [\partial_{z}u - \partial_{z}u_{\infty}]^2 \frac{\big[ (\partial_{z}u)^{-2} - (\partial_{z}u_{\infty})^{-2} \big]}{[\partial_{z}u - \partial_{z}u_{\infty}]} dz.
	\end{aligned}
	\end{equation}
	Then since $\partial_{z}u_{\infty},~\partial_{z}u>0$ and the slope of $f(x)=x^{-2}$ is always negative on $x>0$, we get $I_{1}<0$.
	Note that since the solution has a compact support, the boundary terms vanish in the integration by parts.
	
	For $I_{2}$, we first consider the case where $k=1$. Using the same technique used in \cite{li2004long}, we conclude the following identity
	
	\begin{equation}
	\label{I2}
	I_{2}  = \int_{0}^{1}\int_{0}^{1} [H-K] \big[ G^{\prime}(H) - G^{\prime}(K) \big] d\xi dz,
	\end{equation} 
	in which we have used that 
	\begin{equation}
	\label{zero}
	\int_{0}^{1}[u(z,t)-v(z,t)]dz=0 
	\end{equation}
	for any two solutions $u$ and $v$ of \eqref{assymptotic9} with zero center of mass. Moreover, using \eqref{zero} we also have the following 
	
	\begin{equation}
	\label{equality}
	\int_{0}^{1}\int_{0}^{1} [H-K]^{2}dzd\xi = 2\int_{0}^{1} [u(z,t)-u_{\infty}(z)]^{2}dz.
	\end{equation}
	
	Next we prove that if we take the support of the steady state to be inside the region where the kernel $G$ is strictly concave, assuming as well that the support of the initial datum is close to the steady state, then for all $t>0$ the solution $\rho(t)$ will be very close to the steady state $\rho_{\infty}$ in the $\infty$-Wasserstein distance.
	
	\begin{prop}
		\label{T1}
		Let $\rho(x,t)$ be the solution to \eqref{assymptotic1} having initial density $\rho_{0}\in L^{2}(\R)\cap L_{+}^{1}(\R)$ with unit mass and compact support. Let $\rho_{\infty}(x)$ be the stationary solution to \eqref{assymptotic1} with unit mass and same center of mass of $\rho_0$. Let $u_0$ and $u_\infty$ be the pseudo-inverse variables corresponding to $\rho_0$ and $\rho_\infty$ respectively. Assume
		
		\begin{itemize}
			\item[(i)] $\|u_{\infty}\|_{L^{\infty}([0,1])} < \frac{\lambda}{4}-\delta$,
			\item[(ii)] $ \|u_0-u_{\infty}\|_{L^{\infty}([0,1])}<\frac{\lambda}{4}-\delta$, 
		\end{itemize}
		where $\lambda>0$ is s.t. $G^{\prime\prime}(x)<-c<0$ on $[-\lambda,\lambda]$ for some $c>0$ and $\frac{\lambda}{4}>\delta>0$ is arbitrarily small. Then 
		
		\begin{equation}
		\label{assymptotic10}
		\|u(\cdot,t)-u_{\infty}\|_{L^{\infty}([0,1])} \leq \|u_0-u_{\infty}\|_{L^{\infty}([0,1])},
		\end{equation}
		for all $t\geq 0$.
	\end{prop}
	
	\begin{proof}
		Assume by contradiction that there exists $t^{*}>0$ such that
		\[\|u(t^{*})-u_{\infty}\|_{L^{\infty}([0,1])} > \|u_{0}-u_{\infty}\|_{L^{\infty}([0,1])}.\] 
		Since the $L^{\infty}([0,1])$-norm is the $k\rightarrow +\infty$ limit of $L^{2k}([0,1])$-norm, then for $k \gg 1$, 
		
		\begin{equation}
		\label{assymptotic11}
		\|u(\cdot,t^{*})-u_{\infty}\|_{L^{2k}([0,1])} > \|u_0-u_{\infty}\|_{L^{2k}([0,1])}.
		\end{equation}
		Now, for such values of $k$, let 
		\[\bar{t} = \inf\left\{t\geq 0\quad\text{s.t.}\quad\frac{d}{dt}\|u(.,t)-u_{\infty} \|_{L^{2k}} > 0 \right\}.\] Due to \eqref{assymptotic11}, $\bar{t}<+\infty$. As a consequence, we get
		
		\begin{equation}
		\label{assymptotic12}
		\frac{d}{dt}\|u(.,\bar{t})-u_{\infty}\|_{L^{2k}([0,1])}^{2k} \geq 0.
		\end{equation}
		By \eqref{I1I2} and from $I_1\leq 0$ we get  
		
		\begin{equation*}
		\begin{aligned}
		& \frac{d}{dt}\|u(z,\bar{t})-u_{\infty}(z)\|_{L^{2k}([0,1])}^{2k} \leq 
		\\
		&\qquad\qquad 2k \int\int [u(z,\bar{t})-u_{\infty}(z)]^{2k-1}\big[G^{\prime}\big(u(z,\bar{t})-u(\xi,\bar{t})\big) - G^{\prime}\big(u_{\infty}(z)-u_{\infty}(\xi)\big)\big]dzd\xi
		\\
		&\qquad\qquad := I_{2}\big|_{t=\bar{t}}.
		\end{aligned}
		\end{equation*}
		Changing the role of $\xi$ and $z$ in $I_{2}\big|_{t=\bar{t}}$, we get
		
		\begin{equation*}
		\begin{aligned}
		I_{2}\big|_{t=\bar{t}} & = - 2k \int\int [u(\xi,\bar{t})-u_{\infty}(\xi)]^{2k-1}\big[G^{\prime}\big(u(z,\bar{t})-u(\xi,\bar{t})\big)  - G^{\prime}\big(u_{\infty}(z)-u_{\infty}(\xi)\big)\big]dzd\xi,
		\end{aligned}
		\end{equation*}
		so we have 
		
		\begin{equation*}
		\begin{aligned}
		I_{2}\big|_{t=\bar{t}} & = k \int\int \Big( [u(z,\bar{t})-u_{\infty}(z)]^{2k-1} - [u(\xi,\bar{t})-u_{\infty}(\xi)]^{2k-1} \Big) 
		\\
		& \qquad \Big(G^{\prime}\big(u(z,\bar{t})-u(\xi,\bar{t})\big) - G^{\prime}\big(u_{\infty}(z)-u_{\infty}(\xi)\big)\Big)dzd\xi
		\\
		& = k \int\int \Big( [u(z,\bar{t})-u_{\infty}(z)]^{2k-1} - [u(\xi,\bar{t})-u_{\infty}(\xi)]^{2k-1} \Big)
		\\
		& \qquad \Big( [u(z,\bar{t})-u_{\infty}(z)] - [u(\xi,\bar{t})-u_{\infty}(\xi)] \Big)
		\\
		& \qquad \Bigg( \frac{\big[G^{\prime}\big(u(z,\bar{t})-u(\xi,\bar{t})\big)-G^{\prime}\big(u_{\infty}(z)-u_{\infty}(\xi)\big)\big]}{[u(z,\bar{t})-u_{\infty}(z)] - [u(\xi,\bar{t})-u_{\infty}(\xi)]} \Bigg)dzd\xi.
		\end{aligned}
		\end{equation*}
		It follows from \eqref{assymptotic12} and $I_{1}\leq0$ that $I_{2}\geq0$. As a consequence, for some $\xi,$ $z\in[0,1]$, we get
		
		\begin{equation}
		\frac{\big[G^{\prime}\big(u(z,\bar{t})-u(\xi,\bar{t})\big)-G^{\prime}\big(u_{\infty}(z)-u_{\infty}(\xi)\big)\big]}{[u(z,\bar{t})-u_{\infty}(z)] - [u(\xi,\bar{t})-u_{\infty}(\xi)]} \geq 0.
		\end{equation}
		Which implies that there exists $\xi,$ $z$ s.t. $|u(z,\bar{t})-u(\xi,\bar{t})|>\lambda$ and this is because $G$ is strictly concave on $[-\lambda,\lambda]$. This gives 
		
		\begin{equation}
		\label{14}
		\|u(\bar{t})\|_{L^{\infty}([0,1])} \geq \frac{1}{2}\big|u(z,\bar{t})-u(\xi,\bar{t})\big| > \frac{\lambda}{2}.
		\end{equation}
		So using (i) we obtain the following strict inequality
		
		\begin{equation}
		\label{14a}
		\|u(\bar{t}) - u_{\infty}\|_{L^{\infty}([0,1])} > \frac{\lambda}{4} + \delta.
		\end{equation}
		On the other hand, since $\bar{t}$ is the infimum of all times in which $\|u(t) - u_{\infty}\|_{L^{2k}([0,1])}$ starts increasing, we have by (ii),
		
		\begin{equation}
		\label{14b}
		\|u(\bar{t}) - u_{\infty}\|_{L^{2k}([0,1])} \leq \frac{\lambda}{4} - \delta.
		\end{equation}
		Finally, since \eqref{14b} holds for $k\gg 1$ and we know that $\lim_{k\rightarrow+\infty} \|u(\bar{t}) - u_{\infty}\|_{L^{2k}([0,1])} = \|u(\bar{t}) - u_{\infty}\|_{L^{\infty}([0,1])}$, we obtain
		
		\begin{equation}
		\label{14c}
		\|u(\bar{t}) - u_{\infty}\|_{L^{\infty}([0,1])} \leq \frac{\lambda}{4} - \delta.
		\end{equation}
		which is a contradiction with \eqref{14a}.
	\end{proof}
	
	We are now ready to prove our main result of local asymptotic stability of steady states. 
	
	\begin{teo}\label{mainthm}
		Let $\rho(x,t)$ be the solution to \eqref{assymptotic1} having initial density $\rho_{0}\in L^{2}(\R)\cap L_{+}^{1}(\R)$ with unit mass and compact support. Let $\rho_{\infty}(x)$ be the stationary solution to \eqref{assymptotic1} with unit mass and same center of mass of $\rho_{0}$. Assume $G\in C^2$ and $G''<-c<0$ on the interval $[-\lambda,\lambda]$ for some constant $c>0$. Suppose that there exists $\frac{\lambda}{4}>\delta>0$ such that 
		\begin{itemize}
			\item [(i)] The support of $\rho_\infty$ is contained in $[-\frac{\lambda}{4} +\delta,\frac{\lambda}{4}-\delta]$,
			\item [(ii)] $W_\infty(\rho_0,\rho_\infty)\leq \frac{\lambda}{4}-\delta$.
		\end{itemize}
		Then,
		\begin{equation}
		\label{assymptotic15}
		W_{2}(\rho(t),\rho_{\infty}) \leq W_{2}(\rho(0),\rho_{\infty}) e^{-ct},
		\end{equation}
		for all $t \geq 0$.
	\end{teo}
	
	\begin{proof}
		In view of \eqref{I1I2}, \eqref{I1}, and \eqref{I2}, we have
		
		\begin{equation}
		\label{assymptotic16}
		\frac{d}{dt} \| u(z,t)-u_{\infty}(z)\|_{L^{2}}^{2} \leq \int_{0}^{1}\int_{0}^{1} [H-K] \big[ G^{\prime}(H) - G^{\prime}(K) \big] d\xi dz.
		\end{equation}
		Now by condition (i) we have that $|K|<\frac{\lambda}{2}-2\delta$. Moreover, the condition (ii) and the result in Proposition \ref{T1} imply $|H|\leq2\|u(\cdot,t)-u_{\infty}(\cdot)\|_{L^{\infty}} + |K|<\lambda - 4\delta$. This guarantees that $\frac{G^{\prime}(H) - G^{\prime}(K)}{H-K} < - c $ for some constant $c>0$ due to the strict concavity of $G$ on $[-\lambda,\lambda]$. Thus, using \eqref{equality}, \eqref{assymptotic16} becomes
		
		\begin{equation}
		\label{assymptotic17}
		\begin{aligned}
		\frac{d}{dt} \| u(z,t)-u_{\infty}(z)\|_{L^{2}}^{2} & \leq -c \int_{0}^{1}\int_{0}^{1} [H-K]^{2} d\xi dz 
		\\
		& =  - 2c \int_{0}^{1} [u(z,t)-u_{\infty}(z)]^2 dz.
		\end{aligned}
		\end{equation}
		Then the assertion follows by the Gronwall lemma.
	\end{proof}
	
	\section{Large time decay}\label{sec:arguments}
	
	The nonexistence of a nontrivial steady state for $\varepsilon\geq \|G\|_{L^1}$ is reasonably seen as a consequence of a stronger impact of the diffusion term on the dynamics of \eqref{assymptotic1} compared to the case $0<\varepsilon<\|G\|_{L^1}$. For this reason, in the case $\varepsilon\geq \|G\|_{L^1}$ we expect solutions to decay to zero for large times in a diffusive fashion, similarly to what happens to the solution to the Cauchy problem of the porous medium equation, see \cite{vazquez}. This is one of the goals of this subsection.
	
	On the other hand, there is another question which naturally arises in the case $0<\varepsilon<\|G\|_{L^1}$, that is whether or not solutions \emph{may exhibit a diffusive behavior also in the case $0<\varepsilon<\|G\|_{L^1}$}. As specified in the introduction, this is a difficult question even in classical problems such as the (modified) Keller-Segel system with linear diffusion, see \cite{bedrossian2014}. As we will show later on, numerical simulations suggest that solutions may decay to zero provided the initial condition is "spread" enough, i.e. has large enough second moment. In this section we will produce some incomplete mathematical arguments that support such a conjecture.
	
	We start by proving the decay of the solution to equation \eqref{assymptotic1} in case $\varepsilon \geq \|G\|_{L^1}$.
	
	\begin{teo}\label{T2}
		Let $G\in W^{2,\infty}(\R)\cap L^1(\R)$ with $G\geq 0$ and $G(x)=g(|x|)$ with $g'<0$ on $(0,+\infty)$. Let $\varepsilon\geq\|G\|_{L^1}$. Assume $\rho_0\in \mathcal{P}(\R)\cap L^2(\R)$ with unit mass. Then the solution $\rho(\cdot,t)$ to \eqref{assymptotic1} with $\rho_0$ as initial condition satisfies 
		\[\lim_{t\rightarrow +\infty}\rho(x,t)=0,\qquad \hbox{for a.e. $x\in \R$}.\]
	\end{teo}
	
	\begin{proof}
		The time derivative of $E[\rho(\cdot,t)]$ is given by
		
		\begin{equation}
		\frac{d}{dt}E[\rho(\cdot,t)] = -\int \rho |\partial_{x}(\varepsilon\rho - G*\rho)|^{2} dx := -I[\rho(\cdot,t)],
		\end{equation}
		which gives
		\begin{equation}
		\int_0^{t} I[\rho(\cdot,\tau)]\,d\tau = E[\rho_0]-E[\rho(\cdot,t)]\leq \frac{\varepsilon}{2}\int \rho_0^2\, dx + \frac{1}{2}\|G\|_{L^\infty}.
		\end{equation}
		The above implies 
		\[\int_0^{+\infty} I[\rho(\cdot,\tau)]\,d\tau<+\infty.\]
		Therefore, up to a subsequence $\rho_{k} := \rho(t_{k})$ we know that $I[\rho_{k}] \rightarrow 0$ as $k \rightarrow +\infty.$ This means that there exists some constant $C_{1} > 0$ s.t. $I[\rho_{k}]\leq C_{1}.$ Expanding the integral $I[\rho]$, using $\rho\in \mathcal{P}(\R)$, and integrating by parts we get
		
		\begin{align}
		& I[\rho] = \varepsilon^2\int\rho\rho_{x}^{2}dx + \int\rho(G^{\prime}*\rho)^{2}dx + \varepsilon\int\rho^2G''*\rho dx\nonumber\\
		& \geq \varepsilon^2\int\rho\rho_{x}^{2}\, dx -\varepsilon\|G''\|_{L^\infty(\R)}\int\rho^2 \, dx\nonumber\\
		& \geq \varepsilon^2\int\rho\rho_{x}^{2}\, dx -\|G''\|_{L^\infty(\R)}\left(2E[\rho_0]+\|G\|_{L^\infty(\R)}\right).\label{entropy1}
		\end{align} 
		The last inequality follows from the fact that $\frac{d}{dt}E\leq0$ which implies that $\varepsilon\int\rho^{2}\leq 2E[\rho]+\|G\|_{L^{\infty}}\leq 2E[\rho_{0}] + \|G\|_{L^{\infty}}$. 
		From the assumptions on $G$, \eqref{entropy1} implies that there exists a constant $C$ such that
		
		\begin{equation}
		\int (\partial_{x}(\rho_{k}^{\frac{3}{2}}))^{2}dx \leq C
		\end{equation}
		uniformly w.r.t. $k$. Using Nash inequality in one space dimension, namely
		\begin{equation}
		\label{nash}
		\|f\|_{L^{2}}^{3} \leq \widetilde{C} \|f\|_{L^{1}}^{2} \|f_{x}\|_{L^{2}},
		\end{equation}
		with $f = \rho_{k}^{\frac{3}{2}}$ in \eqref{nash}, and using standard $L^p$ interpolation inequalities we obtain
		
		\begin{equation}
		\Bigg( \int \rho_{k}^{3} dx \Bigg)^{2} \leq C_{2}
		\end{equation}
		uniformly w.r.t. $k$ for some constant $C_{2}>0$. This means that $\rho_{k}^{\frac{3}{2}}$ is uniformly bounded in $H^{1}.$ By Sobolev embedding, up to a subsequence $\rho_{k}^{\frac{3}{2}} \rightarrow \widetilde{\rho}^{\frac{3}{2}}$ in $L_{loc}^{2}$ and so $\rho_{k}$ converges $a.e$ to $\widetilde{\rho}.$ Since $I[\rho_{k}] \rightarrow 0$ as $k \rightarrow +\infty$ we have that $\liminf_{k\rightarrow+\infty}I[\rho_{k}] = 0.$ Now, since $\rho_k^{3/2}$ converges to $\tilde{\rho}^{3/2}$ weakly in $H^1$ and $\rho_k$ converges strongly in $L^3$ to $\tilde{\rho}$ on compact intervals, for an arbitrary interval $\tilde{I}=[-R,R]$ we have
		\begin{align*}
		& \frac{4 \varepsilon^2}{9}\int_{\tilde{I}} (\partial_x(\tilde{\rho}^{3/2}))^2\, dx + \int_{\tilde{I}}\tilde{\rho}(G'*\tilde{\rho})^2\, dx -2\varepsilon\int_{\tilde{I}}\tilde{\rho}\tilde{\rho}_xG'*\tilde{\rho}\, dx \\
		& \ \leq \liminf_{k\rightarrow +\infty} \left(\frac{4 \varepsilon^2}{9}\int_{\tilde{I}} (\partial_x(\tilde{\rho}_k^{3/2}))^2\, dx + \int_{\tilde{I}}\rho_k(G'*\rho_k)^2\, dx -2\varepsilon\int_{\tilde{I}}\rho_k (\rho_k)_x G'*\rho_k\, dx\right)\\
		& \ = \liminf_{k\rightarrow +\infty} \int_{\tilde{I}} \rho_k\left|\partial_x (\varepsilon \rho_k - G*\rho_k)\right|^2\, dx = \liminf_{k\rightarrow +\infty} I[\rho_k] = 0.
		\end{align*}
		This shows that $I[\widetilde{\rho}] = 0$, which implies that $\int \tilde{\rho}(\varepsilon\tilde{\rho} - G*\tilde{\rho})_{x}^{2} dx = 0$ and so $(\varepsilon\tilde{\rho} - G*\tilde{\rho})_{x} = 0$ on the support of $\tilde{\rho}$. This means that $\tilde{\rho}$ is a steady state. Since we are in the case $\varepsilon \geq \|G\|_{L^1}$ where we have no steady state but zero (see \cite{BurDiFFra12}), we have $\widetilde{\rho} = 0$. As a trivial consequence of the above procedure, every family $\rho(\cdot,t_k)$ with $t_k\rightarrow +\infty$ has a subsequence that converges to zero almost everywhere w.r.t. $x\in\R$. By a.e. uniqueness of the a.e. limit, the whole family $\{\rho(\cdot,t)\}_{t\geq 0}$ converges to zero as $t\rightarrow +\infty$.
	\end{proof}
	
	We now focus on the case $0<\varepsilon<\|G\|_{L^1}$. We consider two arguments suggesting that the large time decay holds also in this case for suitable initial conditions. 
	
	It is well known that the one dimensional porous medium equation
	\[\rho_t =\frac{1}{2}(\rho^2)_{xx}\]
	can be approximated as $N\rightarrow +\infty$ by the empirical measure of the $N$-particle system
	\begin{equation}
	\begin{cases}
	\dot{X}_i(t)=-\frac{N}{2}\left[\left(\frac{1}{X_{i+1}(t)-X_i(t)}\right)^2-\left(\frac{1}{X_{i}(t)-X_{i-1}(t)}\right)^2\right] & i=2,\ldots,N-1,\\
	\dot{X}_1(t)=-\frac{N}{2}\left(\frac{1}{X_2(t)-X_1(t)}\right)^2, & \\
	\dot{X}_N(t)=\frac{N}{2}\left(\frac{1}{X_N(t)-X_{N-1}(t)}\right)^2, &
	\end{cases}
	\end{equation}
	see for instance \cite{gosse2006identification}. As a toy model for \eqref{assymptotic1}, we therefore consider the following two-particle model
	\begin{equation}
	\begin{cases}
	\label{ode3}
	\dot{X}_{1} = -2\varepsilon R_{1}^{2} + \frac{1}{2}G^{\prime}(X_{1} - X_{2})
	\\
	\dot{X}_{2} =  2\varepsilon R_{1}^{2} + \frac{1}{2}G^{\prime}(X_{2} - X_{1})
	\\[0.2cm]
	R_{1} = \frac{1}{2(X_{2} - X_{1})}
	\end{cases}
	\end{equation}
	Assuming the two particles occupy symmetric positions w.r.t. zero, i.e. $X_{1} = -X_{2} = -X,$ then \eqref{ode3} becomes the one ODE
	
	\begin{equation}\label{ODE2}
	\begin{cases}
	\dot{X}(t) = \frac{\varepsilon}{8X^{2}} + \frac{1}{2}G^{\prime}(2X)
	\\
	X(0) = X_{0}
	\end{cases}
	\end{equation}
	Assuming that $-G'$ is zero at $x=0$, it increases on an interval $[0,\ell]$ and it decreases to zero on $(\ell,+\infty)$, with fast enough decay (for instance exponentially), for small $\varepsilon$ it is easy to show the existence of two equilibrium points $a,b$ for the ODE in \eqref{ODE2}, $a$ stable and $b$ unstable. Hence, for the Cauchy problem \eqref{ODE2} one can show the following behavior:
	
	\begin{itemize}
		\item $X_{0}\in(b,+\infty)$ \qquad $\Longrightarrow\quad \lim_{t\rightarrow +\infty} X(t)=+\infty,$
		\item $X_{0}\in[0,b)$ \qquad $\Longrightarrow\quad\lim_{t\rightarrow +\infty} X(t)=a$.
	\end{itemize}
	This means that the behavior of the discrete density 
	\[\rho(t)=\frac{1}{X_2(t)-X_1(t)}=\frac{1}{2X(t)}\] 
	varies for different choices of the initial datum. If we take $X_{0} > b$ then $X(t)$ increases, which means that the support of the density $\rho(t)$ will increase to $+\infty$ as $t\rightarrow+\infty$ i.e. the density $\rho(t)$ decays to zero, see Figure \ref{figure_particles}.  
	
	\begin{figure}
		\begin{center}
			\resizebox{9cm}{5.2cm}{ \begin{tikzpicture}[
				scale=0.8, %
				axis/.style={thick, ->, >=stealth'},
				important line/.style={thick},
				dashed line/.style={dashed, thin},
				pile/.style={thick, ->, >=stealth', shorten <=2pt, shorten
					>=2pt},
				every node/.style={color=black}
				]
				
				\draw[axis] (-0.6,0) -- (8,0) node(xline)[right]  {$t$};
				\draw[axis] (0,-0.6) -- (0,6) node(yline)[above]  {$ X(t)$};
				\draw[thick,dashed] (-0.2,1.5) -- (8,1.5);
				\draw[thick,dashed] (-0.2,3.5) -- (8,3.5);
				
				\draw[blue,thick] (0,0.7) .. controls (2,0.9) and (5,1.1) .. (8,1.3);
				\draw[blue,thick] (0,2.5) .. controls (2,2.2) and (5,1.9) .. (8,1.7);
				\draw[green,thick] (0,4.5) .. controls (2,4.7) and (5,5.1) .. (8,5.4);
				
				
				\node[red] () at (-0.45,1.5) {$a$};
				\node[red] () at (-0.45,3.5) {$b$};
				
				\node () at (-0.55,0.7)  {$X_{0}$};
				\node () at (-0.55,2.5)  {$X_{0}$};
				\node () at (-0.55,4.5)  {$X_{0}$};
				
				\filldraw[red] (0,0.7) circle (2pt) node[anchor=west]{};
				\filldraw[red] (0,2.5) circle (2pt) node[anchor=west]{};
				\filldraw[red] (0,4.5) circle (2pt) node[anchor=west]{};
				
				\end{tikzpicture}}
		\end{center}
		\caption{Two different asymptotic behaviors of the solution depending on the initial datum.}\label{figure_particles}
	\end{figure}
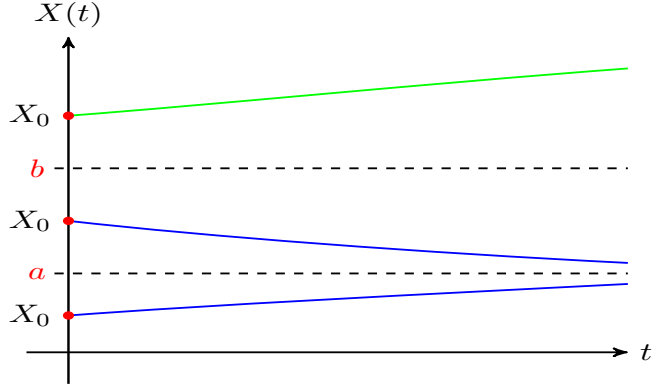
	
	Next we use another argument at the continuum level to show different behaviors of the solution to equation \eqref{assymptotic1} for different choices of initial data for short times. Consider equation \eqref{assymptotic1}, we compute the time derivative of the second moment
	\[M_2(t)=\int_\R |x|^2 \rho(x,t)\, dx\]
	as follows:
	\begin{equation}
	\label{2nd moment}
	\begin{aligned}
	\frac{d}{dt}M_2(t)&=\int_\R|x|^2 \rho_t\, dx = \int_\R|x|^2\left(\frac{\varepsilon}{2}(\rho^2)_{x}-\rho G'*\rho\right)_x\, dx\\
	& = \varepsilon\int_\R \rho(x,t)^2 \, dx +\iint_{\R\times \R}(x-y)G'(x-y)\rho(x,t)\rho(y,t)\, dy\,dx.
	\end{aligned}
	\end{equation}
	Then we consider the first initial datum
	\[\rho_R(x)=\begin{cases}
	\frac{1}{2R} & \hbox{if $-R\leq x\leq R$}\\
	0 & \hbox{otherwise}.
	\end{cases}\]
	By substituting the explicit initial condition $\rho_R$, after few computations we get
	\begin{align*}
	& \frac{d}{dt}M_2(t)\Big|_{t=0} = \frac{\varepsilon}{2R}+\frac{1}{4R^2}\iint_{[-R,R]\times [-R,R]}(x-y)G'(x-y)\, dx\, dy\\
	& \ \geq  \frac{\varepsilon}{2R}+\frac{1}{4R^2}\iint_{\{-2R\leq x-y\leq 2R\,,\,\,-2R\leq x\leq 2R\}}(x-y)G'(x-y)\, dx\, dy\\
	& \ = \frac{1}{2R}\left(\varepsilon +2\int_0^{2R} z G'(z)\, dz\right) =  \frac{1}{2R}\left(\varepsilon -\int_{-2R}^{2R} G(z)\, dz + 4R\,G(2R)\right).
	\end{align*}
	We recall that $zG'(z)\leq 0$ on $z\geq 0$. Assuming
	\[\varepsilon<\|G\|_{L^1}\]
    and that $G$ decays to zero at infinity faster than $1/R$,
	there exists a critical value $R_0>0$ such that
$0< R< R_0$ implies that second moment of the corresponding solution has a positive derivative at $t=0$. On the other hand, if $R>R_0$ then the time derivative of the second moment is negative. This heuristic computation recalls a behavior similar to the one described in the above two-particle toy model, in which the solution tends to stabilize around a confined equilibrium.

In order to obtain a situation fully similar to the one in the particle toy-model we need to find an initial condition with large second moment yielding an increasing second moment at the initial time. To perform this task we need to include \emph{oscillations} in the initial condition too. Let us consider the initial datum
	\begin{equation}
	\label{Initial}
	\rho_{\delta}(x)=\frac{2\delta}{
		\sqrt{\pi}(1+e^{-\frac{1}{\delta^2}})}e^{-(\delta x)^2}\cos^2(x), \qquad\delta>0.
	\end{equation}
We can prove that for $\delta>0$ small enough (hence, high second moment) and with Gaussian kernel the second moment of the corresponding solutions grows at $t=0$. In order to see that, we first compute the Fourier transform of the initial condition
\begin{align*}
 & \hat{\rho}_\delta(\xi)=\frac{1}{1+e^{-\frac{1}{\delta^2}}} e^{-\frac{-\pi^2 \xi^2}{\delta^2}}\left(e^{\frac{1}{\delta^2}}+\cosh\left(\frac{2\pi\xi}{\delta^2}\right)\right).
\end{align*}
Assuming that $G$ is the Gaussian kernel 
	\[G(x)=\frac{1}{\sqrt{\pi}}e^{-x^2},\]
   we compute
  \[\widehat{-x G'(x)}(\xi)=-e^{-\pi^2\xi^2}(2\pi^2\xi^2-1).\]
Using Plancherel's theorem we can rewrite \eqref{2nd moment} in terms of the Fourier transform as follows
	as the following
	\begin{equation}
	\label{2nd moment FT}
	\begin{aligned}
	\frac{d}{dt}M_2(t)\Big|_{t=0}& = \varepsilon\int_\R \rho_\delta(x,t)^2 \, dx +\iint_{\R\times \R}(x-y)G'(x-y)\rho_\delta(x,t)\rho_\delta(y,t)\, dy\,dx
	\\
	& = \int_\R\big(\varepsilon +e^{-\pi^2\xi^2}(2\pi^2\xi^2-1)\Big) \widehat{\rho}_\delta^2(\xi) dx.
	\end{aligned}
	\end{equation}
A tedious but simple computation yields
	\begin{equation}
	\label{2nd moment decay}
	\begin{aligned}
	\frac{d}{dt}M_2(t)\big|_{t=0}& = \frac{\varepsilon \delta}{2\sqrt{2\pi}}\Bigg( \frac{1}{(1+e^{\frac{1}{\delta^2}})^2}+ \frac{4e^{\frac{3}{2\delta^2}}}{(1+e^{\frac{1}{\delta^2}})^2}  + \frac{3e^{2\frac{1}{\delta^2}}}{(1+e^{\frac{1}{\delta^2}})^2} \Bigg)
	\\
	& +  \frac{\delta}{\sqrt{\pi}(\delta^2+2)^{\frac{3}{2}}}\Bigg( \bigg(\frac{4}{\delta^2+2} - 1\bigg)\cdot\frac{e^{\frac{2}{\delta^2+2}\cdot 2\frac{1}{\delta^2}}}{(1+e^{\frac{1}{\delta^2}})^2}  +  \bigg(\frac{4}{\delta^2+2} - 4\bigg)\cdot\frac{e^{(1+\frac{1}{\delta^2+2})\cdot\frac{1}{\delta^2}}}{(1+e^{\frac{1}{\delta^2}})^2}  
	\\
	& -  \frac{1}{(1+e^{\frac{1}{\delta^2}})^2} -  \frac{2e^{2\frac{1}{\delta^2}}}{(1+e^{\frac{1}{\delta^2}})^2}   \Bigg).
	\end{aligned}
	\end{equation} 
	For small $\delta$, \eqref{2nd moment decay} becomes
	\[\frac{d}{dt}M_2(t)\big|_{t=0} \cong \frac{3\varepsilon\delta}{\sqrt{8\pi}} - \frac{\delta}{\sqrt{8\pi}} = \frac{\delta}{\sqrt{8\pi}}(3\varepsilon - 1). \]
	We conclude that $\frac{d}{dt}M_{2}(t)|_{t=0}>0$  for small $\delta$ and for $\varepsilon>1/3$. We observe that such a range of $\epsilon$ implies the existence of non trivial steady states. We stress once again that this computation is merely heuristic. However, the very last simulation performed in the next section (see figure \ref{fig.Decay-Four-Figures}) suggests that such an initial condition produces a dominant repulsive effect. The fact that an oscillating behavior "empowers" the diffusive effects is reminiscent of the classical behavior of a linear reaction diffusion equation, in which highly oscillating initial conditions can compensate large reaction rates and produce a diffusive decay for large times.

	
	\section{Numerical Simulations}\label{sec:numerics}
	In this section we shall present three numerical simulations in which we validate our results about the convergence to the non trivial steady state, decay of the solution to zero in the diffusion-dominated regime $\varepsilon\geq\|G\|_{L{1}}$, and the growing of the second moment of the solution in the case of the special initial condition \eqref{Initial}. In the first two simulations we use two different methods, the finite volume method introduced in \cite{carrillo2015finite} and the particle method already sketched in the previous section. In the last simulation we use only the particle method.
	\\
	
	We begin with the finite volume method. We apply a $1$D positive preserving finite-volume method for \eqref{assymptotic1}, see the paper \cite{carrillo2015finite}. Divide the computational domain into finite-volume cells $U_{i}=[x_{i-\frac{1}{2}},x_{i+\frac{1}{2}}]$ of a uniform size $\Delta x$ with $x_{i} = i\Delta x$, $i \in\{-m,~\dots,m\}$. Let
	\begin{equation*}
	\overline{\rho}_{i}(t) = \frac{1}{\Delta x}\int_{U_{i}}\rho(x,t)dx,
	\end{equation*}
	denote the averages of the solution $\rho$ computed at each cell $U_{i}$. Then integrating equation \eqref{assymptotic1} over each cell $U_{i}$, we obtain a semi-discrete finite-volume scheme given by the following system of ODEs for $\overline{\rho}_{i}$
	
	\begin{equation}
	\label{ode1}
	\frac{d\overline{\rho}_{i}(t)}{dt} = -\frac{F_{i+\frac{1}{2}}(t) - F_{i-\frac{1}{2}}(t) }{\Delta x},
	\end{equation}
	where the numerical flux $F_{i+\frac{1}{2}}$ is an approximation for our continuous flux $-\rho(\varepsilon\rho - G*\rho)_{x}$. We obtain the following expression for $F_{i+\frac{1}{2}}$
	
	\begin{equation}
	F_{i+\frac{1}{2}} = \max (u_{i+1},0)\Big[ \overline{\rho}_{i} + \frac{\Delta x}{2}(\rho_{x})_{i} \Big] + \min (u_{i+1},0)\Big[ \overline{\rho}_{i} - \frac{\Delta x}{2}(\rho_{x})_{i} \Big]
	\end{equation}
	where 
	
	\begin{equation}
	u_{i+1} = \sum_{j}\overline{\rho}_{j}\big( G(x_{i+1} - x_{j}) - G(x_{i} - x_{j}) \big) - \frac{\varepsilon}{\Delta x} \big( \overline{\rho}_{i+1} - \overline{\rho}_{i} \big)
	\end{equation}
	and
	
	\begin{equation}
	(\rho_{x})_{i} = \mbox{minmod} \Bigg( 2\frac{\overline{\rho}_{i+1} - \overline{\rho}_{i}}{\Delta x},~ \frac{\overline{\rho}_{i+1} - \overline{\rho}_{i-1}}{2\Delta x},~ 2\frac{\overline{\rho}_{i} - \overline{\rho}_{i-1}}{\Delta x} \Bigg)
	\end{equation}
	where the minmod limiter is given by
	
	\begin{equation}
	\mbox{minmod}(a_{1}, a_{2},~\dots) :=
	\begin{cases}
	\min (a_{1}, a_{2},~\dots), \quad \mbox{if} ~ a_{i} > 0\quad \forall i
	\\
	\max (a_{1}, a_{2},~\dots), \quad \mbox{if} ~ a_{i} < 0\quad \forall i
	\\
	0, \qquad\qquad\qquad\mbox{otherwise.}
	\end{cases}
	\end{equation}
	Finally, we integrate the semi-discrete scheme \eqref{ode1}, which is a system of ODEs, numerically using the third-order strong preserving Runge-Kutta (SSP-RK) ODE solver used in \cite{gottlieb2001strong}. 
	\\
	
	The second method is a particle method in which we approximate the PDE \eqref{assymptotic1}, (see \cite{gosse2006identification}), by a system of $N$ particles $X_{1}(t),~\dots,X_{N}(t)$ with equal masses $m_{i}=\frac{1}{N}$
	
	\begin{equation}
	\label{ode2}
	\begin{cases}
	\dot{X}_{i} = \varepsilon N(R_{i-1}^{2} - R_{i}^{2}) + \frac{1}{N}\sum_{k\neq i}G^{\prime}(X_{i} - X_{k}),\quad i = 2,~\dots,~N-1
	\\
	\dot{X}_{1} = -\varepsilon N R_{1}^{2} + \frac{1}{N}\sum_{1<k}G^{\prime}(X_{1} - X_{k})
	\\
	\dot{X}_{N} = \varepsilon N R_{N-1}^{2} + \frac{1}{N}\sum_{k<N}G^{\prime}(X_{N} - X_{k})
	\end{cases}
	\end{equation}
	where
	
	\begin{equation}
	R_{i} = \frac{1}{N(X_{i+1} - X_{i})}.
	\end{equation}
	Then we solve the particle system \eqref{ode2} using the Runge-Kutta MATLAB solver ODE23. Note that the initial mesh sizes are automatically determined by the total number of particles $N$ and the initial density values. We take the initial positions $X(0) = X_{0} = (X_{0}^{1},~X_{0}^{2},~\dots,~X_{0}^{N})$ s.t.
	
	\begin{equation}
	\int_{X_{0}^{i}}^{X_{0}^{i+1}} \rho_{0} dX = \frac{1}{N-1} \qquad\qquad i = 1,~2,~\dots,N-1.
	\end{equation}
	\\
	
	We start now with the simulations of  \eqref{assymptotic1}. In order to show the several behaviors of solutions we always use the same aggregation kernel, that is
 \[ G(x)=\frac{1}{\sqrt{\pi}}e^{-x^2}.\]
   In the first simulation we take the steady state inside the interval where $G$ is concave. This will be ensured by choosing small enough $\epsilon$. Indeed, we recall that the support of the steady state is an increasing function of $\epsilon$, degenerating to a point particle when $\epsilon=0$, see \cite{BurDiFFra12}. More precisely, we choose the value $\varepsilon=0.002$, which guarantees that condition (i) in Theorem \ref{T1} is satisfied. Then we take the initial density in such a way that condition (ii) in Theorem \ref{T1} is also satisfied, and this explains the choices of the initial data in Figure \ref{fig.SteadyState}. Applying both the aforementioned methods (particle method and finite volume) we get the results presented in Figure \ref{fig.SteadyState}, which show convergence to the steady state.
	\begin{figure}[!htbp]
		\hspace*{-0.4cm}\includegraphics[width=12cm,height=8.5cm]{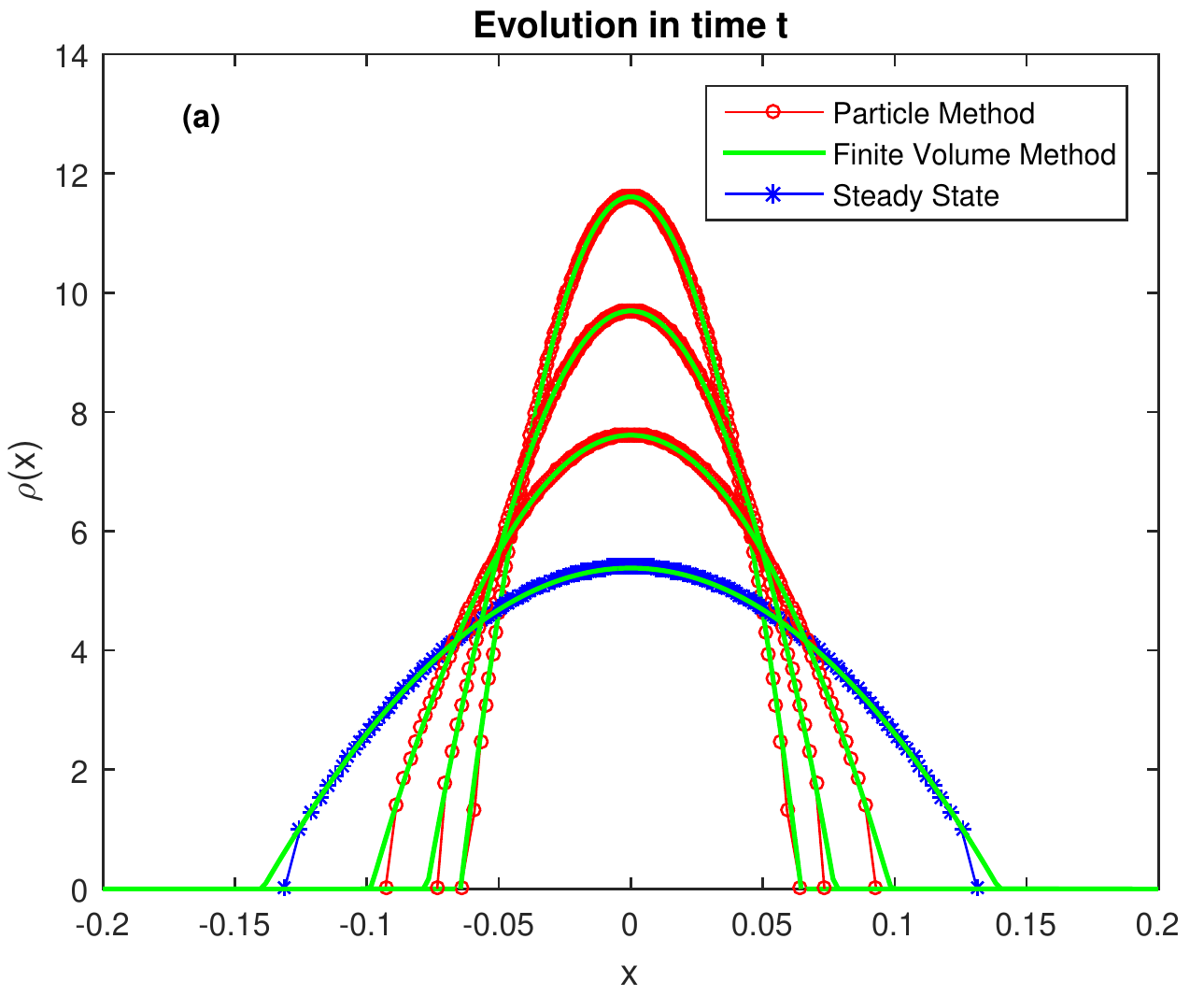}
		\includegraphics[width=12cm,height=8.5cm]{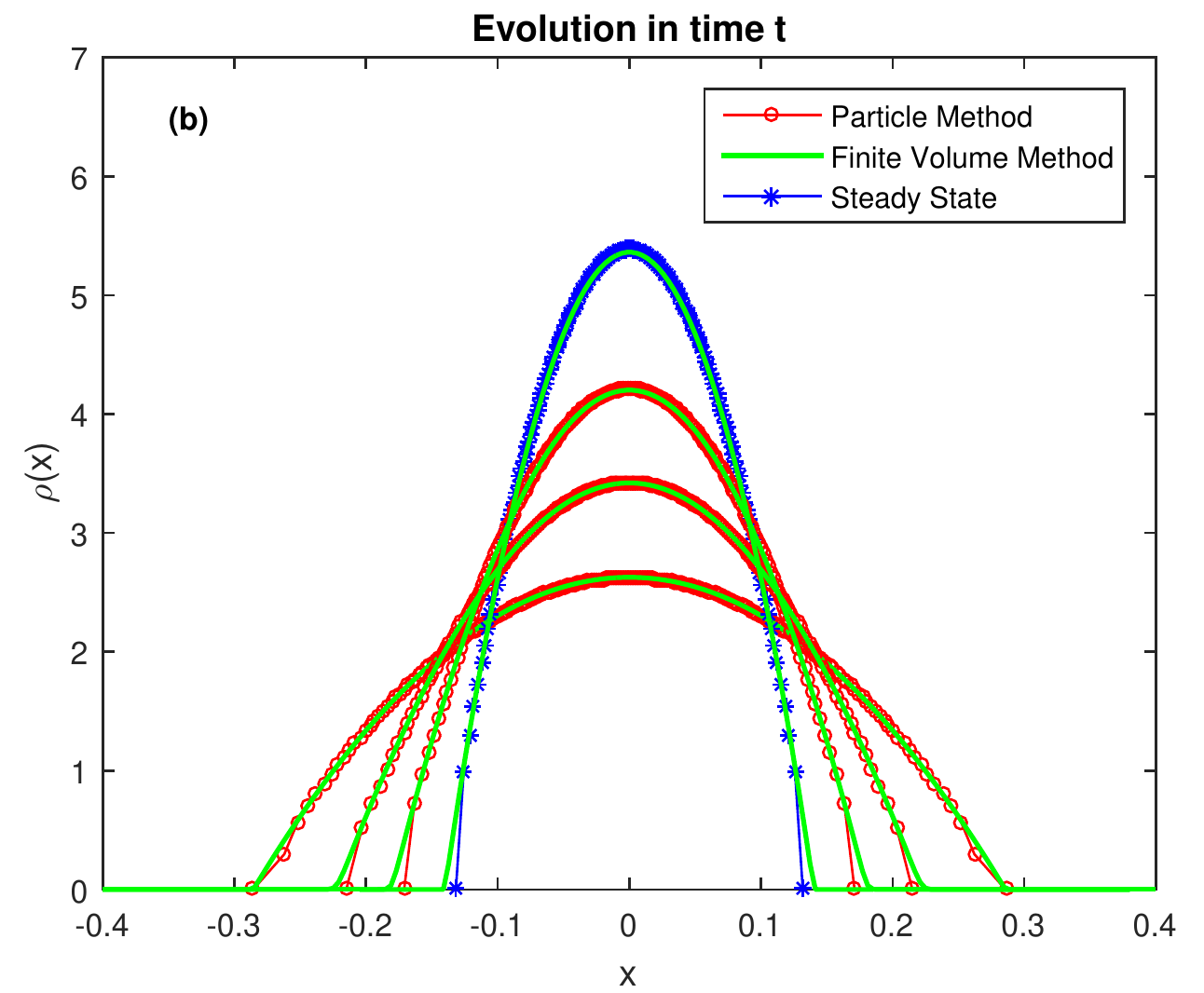}
		\caption{Convergence of the solution to the steady state. $\varepsilon=0.002,$ $G(x)=\frac{1}{\sqrt{\pi}}e^{-x^2}.$ \textbf{(a)} $\rho_{0}(x)= \frac{93}{8}(1-\frac{961}{4}x^2),$ \textbf{(b)} $\rho_{0}(x)= \frac{21}{8}(1-\frac{49}{4}x^2)$}
		\label{fig.SteadyState}
	\end{figure}
	\\
	
The second simulation presented in Figure \ref{fig.DecaytoZero2} shows the decay of the solution to \eqref{assymptotic1} to zero in the diffusion-dominated regime $\varepsilon\geq\|G\|_{L^{1}}$. In this simulation we use the same initial datum $\rho_{0}(x)= \frac{21}{8}(1-\frac{49}{4}x^2)$ as in the first simulation, and we take $\varepsilon=2$.
	\begin{figure}
		\begin{center}
			\includegraphics[width=12cm,height=8.5cm]{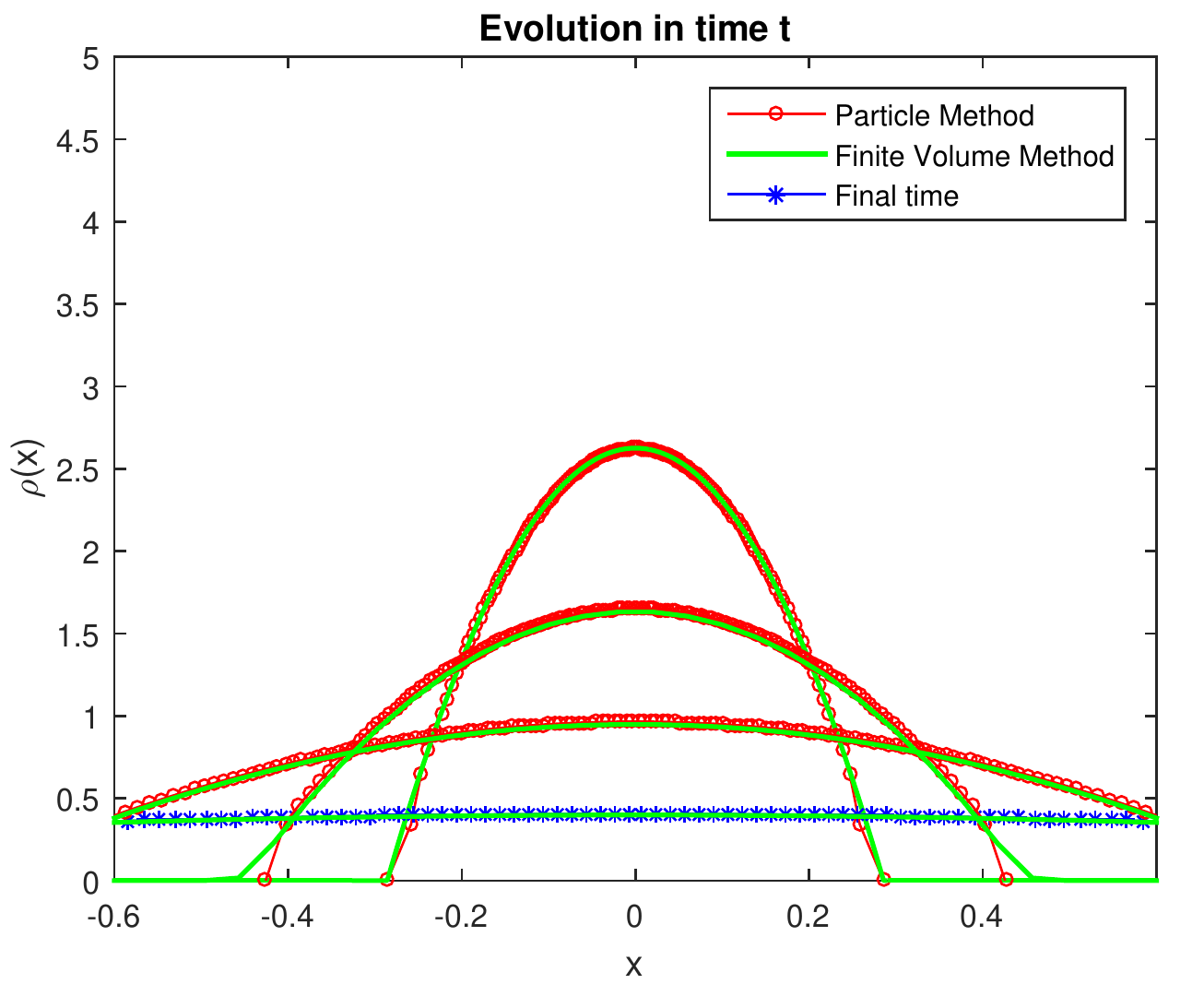}
		\end{center}
		\caption{Decay of the solution to zero in the diffusion-dominated regime $\varepsilon\geq\|G\|_{L^{1}}$. $\varepsilon=2,$ $G(x)=\frac{1}{\sqrt{\pi}}e^{-x^2},$ $\rho_{0}(x)= \frac{21}{8}(1-\frac{49}{4}x^2)$.}
		\label{fig.DecaytoZero2}
	\end{figure}
	\\
	
	Finally, we show that a multiple behavior is possible for a fixed $\epsilon<\|G\|_{L^1}$. More precisely, we consider $\epsilon=0.5$. In Figure \ref{fig.Decay-Four-Figures} we prescribe the initial condition $\rho_\delta$ introduced in \eqref{Initial}, which we proved to yield an initial growth of the second moment. The simulation in Figure \ref{fig.Decay-Four-Figures} suggests that the second moment grows for some time, and the solution takes the shape of many peaks interacting with each other. In figure \ref{fig.SteadyState2} we use the same $\epsilon$ and re-use the same initial data of our first simulation. Although the conditions of Theorem \ref{T1} are not met, we still get convergence towards the steady state. The last two simulations support our conjecture that a multiple behavior holds for \eqref{assymptotic1} in the aggregation-dominated regime, namely $0<\varepsilon<\|G\|_{L^1}$.
	\begin{figure}\label{fig:osc}
		\begin{center}
			\includegraphics[width=16.5cm,height=10cm]{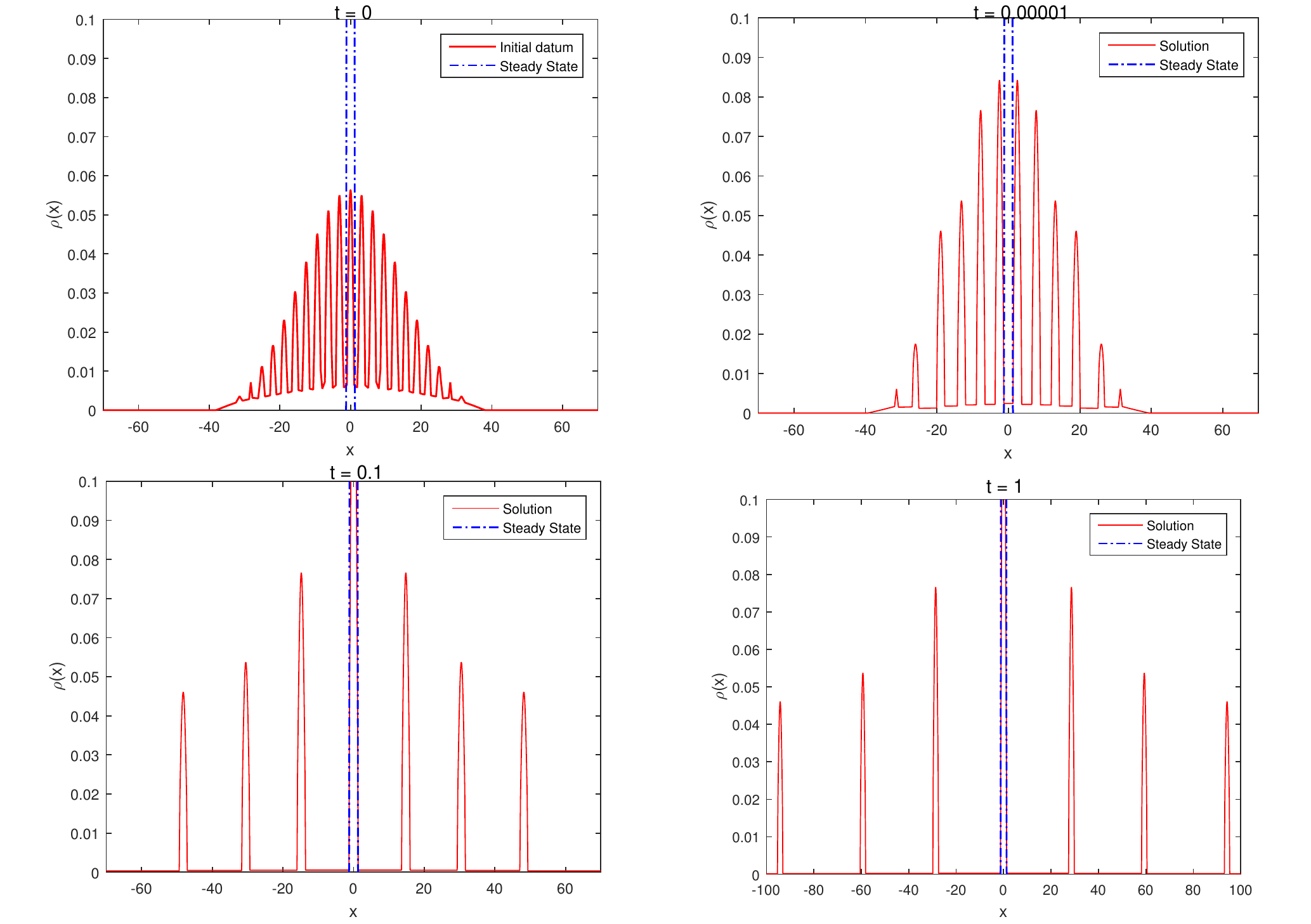}
		\end{center}
		\caption{Growing of the second moment of the solution to \eqref{assymptotic1} in time. $\varepsilon=0.5,$ $G(x)=\frac{1}{\sqrt{\pi}}e^{-x^2}$, and initial datum \eqref{Initial} with $\delta=0.05.$}
		\label{fig.Decay-Four-Figures}
	\end{figure}

    		\begin{figure}[!htbp]
		\hspace*{-0.4cm}\includegraphics[width=12cm,height=8.5cm]{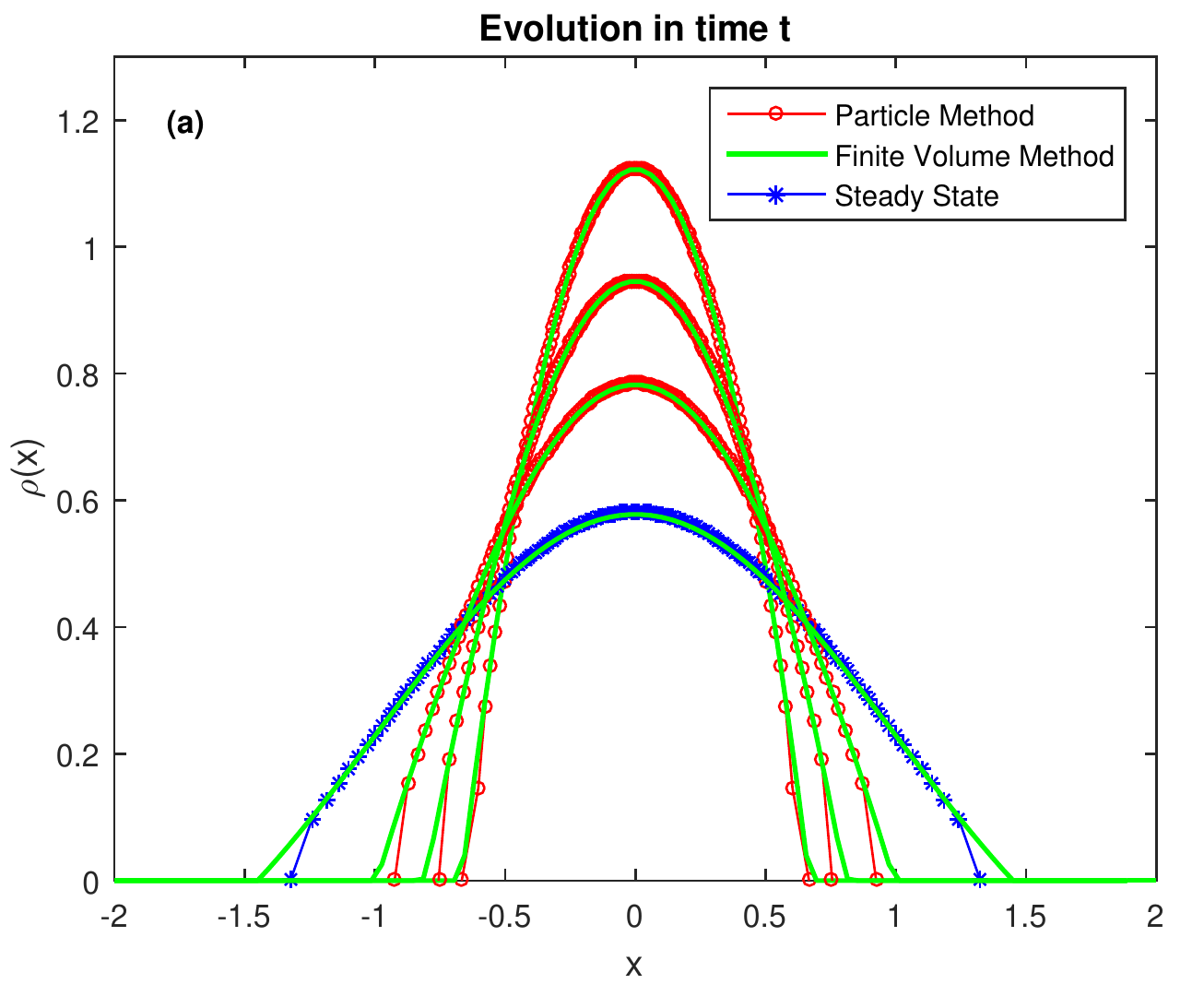}
		\includegraphics[width=12cm,height=8.5cm]{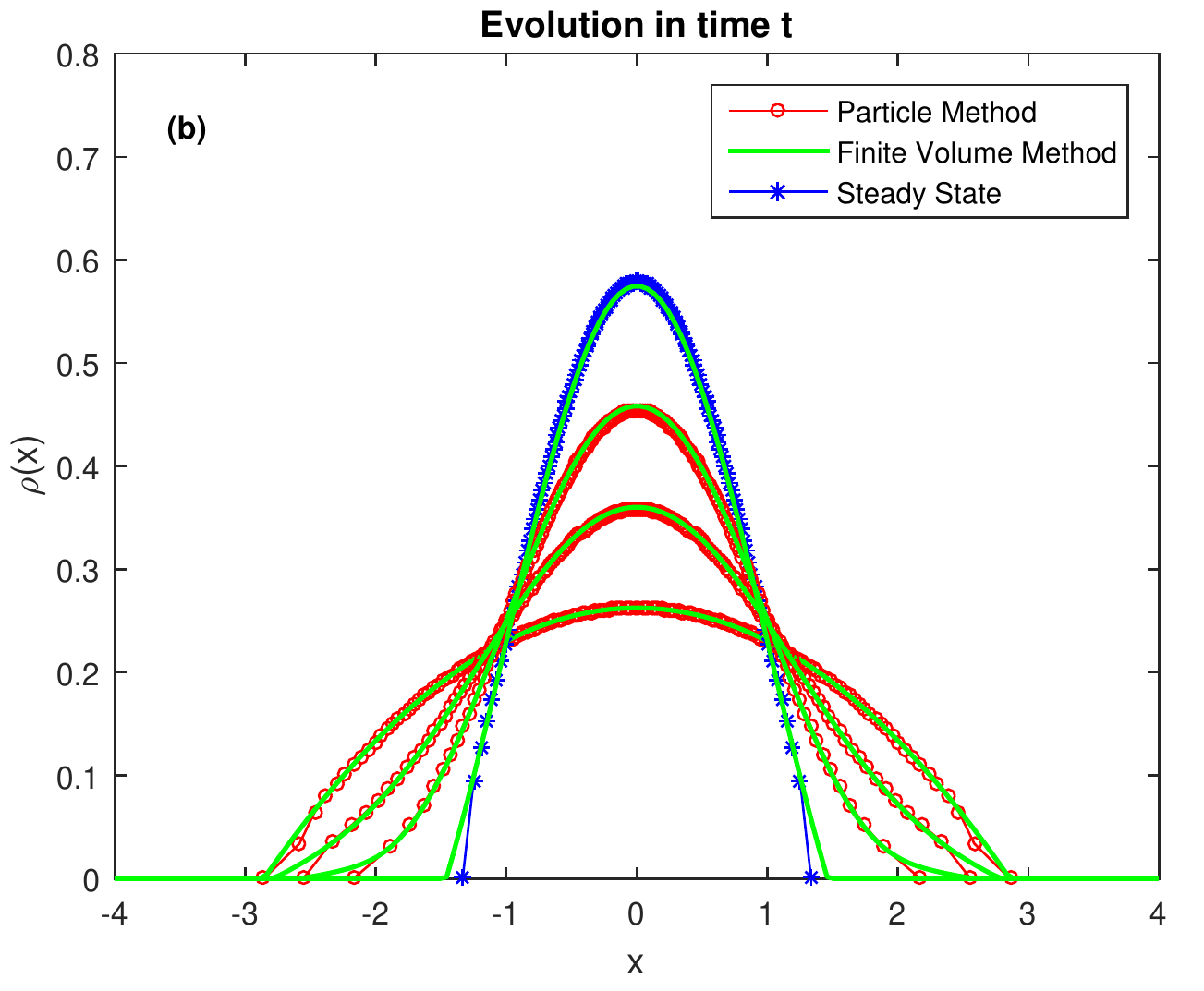}
		\caption{Convergence of the solution to the steady state. $\varepsilon=0.5,$ $G(x)=\frac{1}{\sqrt{\pi}}e^{-x^2}.$ \textbf{(a)} $\rho_{0}(x)= \frac{9}{8}(1-\frac{9}{4}x^2),$ \textbf{(b)} $\rho_{0}(x)= \frac{105}{400}(1-\frac{1225}{10000}x^2)$}
		\label{fig.SteadyState2}
	\end{figure}


	
	

\begin{thebibliography}{10}
		
		\bibitem{AGS}L. Ambrosio, N. Gigli, G. Savar\'{e},
		\textit{Gradient flows in metric spaces and in the space of probability measures.} Lectures in Mathematics ETH $Z\ddot{u}rich$. $Birkh\ddot{a}user$ Verlag, Basel, 2005.
		
		\bibitem{Bedrossian20111927}
		J. Bedrossian.
		\textit{Global minimizers for free energies of subcritical aggregation
			equations with degenerate diffusion.}
		Applied Mathematics Letters, 24(11):1927 -- 1932, 2011.
		
		\bibitem{bedrossian2014}
		J. Bedrossian.
		\textit{Large mass global solutions for a class of L1-critical nonlocal aggregation equations and parabolic-elliptic Patlak-Keller-Segel models.}
		arXiv:1403.4124 2014.
		
        \bibitem{berthelin}
        F. Berthelin, D. Chiron, M. Ribot, 
        \textit{Stationary solutions with vacuum for a one-dimensional chemotaxis model with non-linear pressure.} Commun. Math. Sci., 14(1): 147-186, 2015.
        
		\bibitem{bertozzi1}
		A. L. Bertozzi, T. Laurent, and J. Rosado. 
		\textit{Lp theory for the multidimensional aggregation equation.} Commun. Pure Appl. Math., 64(1):45–83, 2011.
		
		\bibitem{bertozzi2}
		A. Bertozzi and T. Laurent. 
		\textit{Finite-time blow-up of solutions of an aggregation equation in Rn}. Comm. Math. Phys., 274:717–735, 2007.
		
		\bibitem{bertozzi3}
		A. L. Bertozzi and J. Brandman. 
		\textit{Finite-time blow-up of L infinity weak solutions of an aggregation equation.}
		Commun. Math. Sci., 8(1):45–65, 2010.
		
		\bibitem{blanchet_carrillo_laurencot}
		A. Blanchet, J. A. Carrillo, and P. Laurencot. 
		\textit{Critical mass for a Patlak-Keller-Segel model
			with degenerate diffusion in higher dimensions.} 
		Cal. Var. Partial Differential Equations,
		35(2): 133–168, 2009.
		
		\bibitem{blanchet}
		A. Blanchet, J. Dolbeault, and B. Perthame.
		\textit{Two-dimensional {K}eller-{S}egel model: optimal critical mass and qualitative properties of the solutions.}
		Electron. J. Differential Equations (2006), no. 44, p.33.
		
		\bibitem{capasso}
		S. Boi, V. Capasso, and D. Morale.
		\textit{Modeling the aggregative behavior of ants of the species {\it{p}olyergus rufescens}.}
		Nonlinear Anal. Real World Appl. 1 (2000), no. 1, 163-176.
		
		\bibitem{bodnar}
		M. Bodnar, and J. Velazquez. 
		\textit{ An integro-differential equation arising as a limit of individual cell-based
			models.}
		J. Differ. Equations, 222(2):341–380, 2006.
		
		\bibitem{burger2008large}
		M. Burger and M. Di Francesco.
		\textit{Large time behavior of nonlocal aggregation models with nonlinear diffusion.} Networks and Heterogeneous Media, 3(4):749--785, 2008.
		
		\bibitem{BurDiFFra12}
		M. Burger, M. Di Francesco, and M. Franek.
		\textit{Stationary states of quadratic diffusion equations with long-range attraction.}
		Commun. Math. Sci. 11, 3:709--738, 2012.
		
		\bibitem{burger_fetecau_wang}
		M. Burger, R. C. Fetecau, and Y. Huang.
		\textit{Stationary States and Asymptotic Behavior of Aggregation Models with Nonlinear Local Repulsion.}
		SIAM J. Appl. Dyn. Syst., 13(1), 397–424 (2014).
		
		\bibitem{calvez_carrillo_hoffmann}
		V. Calvez, J. A. Carrillo, F. Hoffmann.
		\textit{Equilibria of homogeneous functionals in the fair-competition regime.}
		Nonlinear Analysis TMA 159, 85-128, 2017.
		
		\bibitem{volzone}
		J. A. Carrillo, D. Castorina, and B. Volzone. 
		\textit{Ground states for diffusion dominated free
			energies with logarithmic interaction.}
		SIAM J. Math. Anal., 47(1):1–25, 2015.
		
		\bibitem{carrillo2015finite}
		J. A. Carrillo, A. Chertock, and Y. Huang.
		\textit{A finite-volume method for nonlinear nonlocal equations with a
			gradient flow structure.}
		Communications in Computational Physics, 17(1):233--258, 2015.
		
        \bibitem{gualdani}
        J. A. Carrillo, M. Gualdani, and G. Toscani.
        \textit{Finite speed of propagation in porous media by mass transportation methods.}
        Comptes Rendus Mathematique 338(10): 815-818, 2004.
        
		\bibitem{carrillo_mccann_villani}
		J. A. Carrillo, R. J. McCann, and C. Villani.
		\textit{Kinetic equilibration rates for granular media and related equations: entropy dissipation and mass transportation estimates.}
		Rev. Mat. Iberoam., 19(3):971–
		1018, 2003.
		
		\bibitem{carrillo2005wasserstein}
		J. A. Carrillo and G. Toscani.
		\textit{Wasserstein metric and large-time asymptotics of nonlinear diffusion equations.}
		New trends in mathematical physics, pages 234--244, 2005.
		
		\bibitem{choksi}
		R. Choksi, R. C. Fetecau, I. Topaloglu.
		\textit{On minimizers of interaction functionals with competing attractive and repulsive potentials.}
		Annales de l'Institut Henri Poincare (C) Non Linear Analysis, 32 (6), 1283-1305 (2015).
		
		\bibitem{di2017follow}
		M. Di Francesco, S. Fagioli, M. D. Rosini, and G. Russo.
		\textit{Follow-the-leader approximations of macroscopic models for vehicular and pedestrian flows.}
		In Active Particles, Volume 1, pages 333--378. Springer, 2017.
		
		
		\bibitem{0951-7715-26-10-2777}
		M. Di Francesco and S. Fagioli.
		\textit{Measure solutions for non-local interaction pdes with two species.}
		Nonlinearity, 26(10):2777, 2013.
		
		\bibitem{fellner}
		K. Fellner, G. Raoul,
		\textit{Stability of stationary states of non-local interaction equations.} Mathematical and Computer Modelling 53 (2011) 1436-1450;
		
		\bibitem{gosse2006identification}
		L. Gosse and G. Toscani.
		\textit{Identification of asymptotic decay to self-similarity for
			one-dimensional filtration equations.}
		SIAM Journal on Numerical Analysis, 43(6):2590--2606, 2006.
		
		\bibitem{gottlieb2001strong}
		S. Gottlieb, C. Shu, and E. Tadmor.
		\textit{Strong stability-preserving high-order time discretization methods.}
		SIAM review, 43(1):89--112, 2001.
		
		\bibitem{jager}
		W. J{\"a}ger and S. Luckhaus.
		\textit{On explosions of solutions to a system of partial differential equations modelling chemotaxis.} 
		Trans. Amer. Math. Soc., 329 (1992), no. 2, 819--824.
		
		\bibitem{JKO}R. Jordan, D. Kinderlehrer, F. Otto,
		\textit{The variational formulation of the Fokker-Planck equation.} SIAM J. Math. Anal., 29 (1998), 1-17.
		
		\bibitem{helbing}
		D. Helbing, I. J. Farkas, P. Molnar, and T. Vicsek.
		\textit{Simulation of pedestrian crowds in normal and evacuation situations.} In: M. Schreckenberg and S. D. Sharma (eds.) Pedestrian and
		Evacuation Dynamics (Springer, Berlin), pages 21--58, 2002.
		
		\bibitem{kaib}
		G. Kaib.
		\textit{Stationary States of an Aggregation Equation with Degenerate Diffusion and Bounded Attractive Potential.}
		SIAM J. Math. Anal., 49(1), 272–296 (2017).
		
		\bibitem{kim_yao}
		I. Kim and Y. Yao. 
		\textit{The Patlak-Keller-Segel model and its variations: properties of solutions via maximum principle.} 
		SIAM J. Math. Anal., 44(2):568–602, 2012.
		
		\bibitem{li2004long}
		H. Li and G. Toscani.
		\textit{Long-time asymptotics of kinetic models of granular flows.}
		Archive for rational mechanics and analysis, 172(3):407--428,
		2004.
		
		\bibitem{lieb_yau}
		E. H. Lieb and H.-T. Yau. 
		\textit{The Chandrasekhar theory of stellar collapse as the limit of quantum mechanics.} 
		Commun. Math. Phys., 112(1):147–174, 1987.
		
		\bibitem{lions84}
		P. L. Lions. 
		\textit{The concentration-compactness principle in calculus of variations. The locally compact case, part 1.} 
		Ann. Inst. Henri Poincar´e, Anal. nonlin., 1:109–145, 1984.
		
		\bibitem{mccann}
		R. McCann.
		\textit{A convexity principle for interacting gases.}
		Advances in mathematics, 128, 153-179, 1997.
		
		\bibitem{mogilner}
		A. Mogilner and L. Edelstein-Keshet.
		\textit{A non-local model for a swarm.}
		J. Math. Biol. 38 (1999), no. 6, 534-570.
		
		\bibitem{Morale2005}
		D. Morale, V. Capasso, and K. Oelschl{\"a}ger.
		\textit{An interacting particle system modelling aggregation behavior: from individuals to populations.}
		Journal of Mathematical Biology, 50(1):49--66, 2005.
		
		\bibitem{peletier}L. Scardia, R. Peerlings, M. Geers, and M. A. Peletier, 
		\textit{Mechanics of dislocation pile-ups:
			A unification of scaling regimes.} 
		Journal of the Mechanics and Physics of Solids 70 (2014), 42-61.
		
		\bibitem{russo1990deterministic}
		G. Russo.
		\textit{Deterministic diffusion of particles.}
		Communications on Pure and Applied Mathematics, 43(6):697--733,
		1990.
		
		\bibitem{sznajd}
		K. Sznajd-Weron and J. Sznajd.
		\textit{Opinion evolution in closed community.}
		Int. J. Mod. Phys. C 11 (2000), 1157-1166.
		
		\bibitem{Topaz2006}
		C. M. Topaz, A. L. Bertozzi, and M. A. Lewis.
		\textit{A nonlocal continuum model for biological aggregation.}
		Bulletin of Mathematical Biology, 68(7):1601--1623, 2006.
		
		\bibitem{toscani_granular}
		G. Toscani.
		\textit{Kinetic and hydrodynamic models of nearly elastic granular flows.}
		Monatsh. Math. 142 (2004), no. 1-2, 179--192.
		
		\bibitem{vazquez}
		J. L. Vazquez
		\textit{The Porous Medium Equation.}
		Oxford Mathematical Monographs (2006).
		
		\bibitem{villani}
		C. Villani. 
		\textit{Topics in optimal transportation.} 
		Volume 58 of Graduate Studies in Mathematics. American
		Mathematical Society, Providence, RI, 2003.
		
	\end{thebibliography}
	
	\section*{Acknowledgments}
	MDF acknowledges conversations on this problem with Martin Burger in the past five years. The authors are supported by the local fund of the University of L'Aquila "DP-LAND" (Deterministic Particles for Local And Nonlocal Dynamics), and from the Erasmus Mundus programme "MathMods", www.mathmods.eu. YJ acknowledges support from the Italian INdAM GNAMPA (National group for Mathematical Analysis, Probability, and their applications) project "Analisi di modelli matematici della fisica, della biologia e delle scienze sociali".

\end{document}